\documentclass[11pt,a4paper]{article}
\usepackage{authblk}
\usepackage{amsmath,amsxtra,amscd,amssymb,latexsym,stmaryrd,amsthm}
\usepackage{graphicx,color}
\usepackage{hyperref}
\usepackage{appendix}
\usepackage[utf8]{inputenc}
\usepackage[T1]{fontenc}
\usepackage{mathrsfs}
\usepackage[all]{xy}
\usepackage[centering,a4paper]{geometry}

\allowdisplaybreaks

\definecolor{Red}{cmyk}{0,1,1,0}

\definecolor{Blue}{cmyk}{1,1,0,0}

\theoremstyle{plain}
\newtheorem{theorem}{Theorem}
\newtheorem{corollary}{Corollary}

\theoremstyle{definition}
\newtheorem{definition}[theorem]{Definition}

\hypersetup
{
    pdfauthor={D. B. Aguiar, L. Cioletti and R. Ruviaro},
    pdfsubject={Variational Formulation of Specific Entropy},
    pdftitle={A Variational Principle for the Specific Entropy for Symbolic Systems with Uncountable Alphabets},
    pdfkeywords={Thermodynamic Formalism, Ruelle operator, entropy, equilibrium states}
}

\title{A Variational Principle for the Specific Entropy for Symbolic Systems with Uncountable Alphabets}
\author[$\S$]{D. Aguiar}
\author[$\S$]{L. Cioletti}
\author[$\S$]{R. Ruviaro}
\affil[$\S$]{Universidade de Bras\'ilia, MAT, 70910-900, Bras\'ilia, Brazil}

\date{}
\begin{document}
    \makeatletter
    \def\blfootnote{\gdef\@thefnmark{}\@footnotetext}
    \let\@fnsymbol\@roman
    \makeatother

\maketitle

\begin{abstract}
	In this paper we derived a variational principle for the specific entropy 
	on the context of symbolic dynamics of compact metric space alphabets 
	and use this result to obtain the uniqueness 
	of the equilibrium states associated to a Walters potential.
\end{abstract}

\blfootnote{\textup{2010} \textit{Mathematics Subject Classification}: 37D35.}
\blfootnote{\textit{Keywords}: Thermodynamic Formalism, Ruelle operator, entropy, equilibrium states.}

\section{Introduction}

Let $\mathbb{N}$ be the set of positive integers,
$\mathcal{A}$ a finite alphabet, $\Omega\equiv \mathcal{A}^{\mathbb{N}}$ the product
space equipped with its usual distance $d_{\Omega}$, product topology and 
the sigma-algebra generated by the open sets.  
The dynamics in this paper is given by $\sigma:\Omega\to \Omega$, the left shift mapping. 
In this setting an {\it equilibrium state} for a continuous potential $f:\Omega\to\mathbb{R}$ 
is an element of $\mathscr{M}_{\sigma}(\Omega)$, 
the set of all shift-invariant Borel probability measures,
solving the following variational problem 
\begin{align}\label{prob-var-A-finito}
\sup_{\mu\in\mathscr{M}_{\sigma}(\Omega)} \{H(\mu)+\int_{\Omega} f\, d\mu \},
\end{align}
where $H(\mu)$ is the Kolmogorov-Sinai entropy of $\mu$.
This variational principle was introduced by Ruelle \cite{MR0217610} 
in the context of Statistical Mechanics and later Walters \cite{MR0390180} 
considered this problem in the Ergodic Theory setting. 
This classical problem still is a central one 
in Ergodic Theory/Thermodynamic Formalism and the complete 
classification of all continuous potentials
for which the problem \eqref{prob-var-A-finito} has a unique solution remains open.

One of the goals of this paper is to establish the uniqueness of the equilibrium states for 
H\"older and Walters potentials (see Definition \ref{def-walters-space} for the former)
in a more general setting, where the alphabet $\mathcal{A}$
is a general compact metric space, 
thus including cases where $\mathcal{A}$ is uncountable. 
We shall remark that this uniqueness result is well-known in the context of finite alphabets,
see \cite{MR1783787,MR1793194,MR1085356} and references therein.
The first step towards this uniqueness result on general compact metric space
alphabets was given in \cite{MR3377291}. They proved that any probability measure $\mu$
satisfying $\mathscr{L}^{*}_{\bar{f}}\mu=\mu$ (see Section \ref{sec-main-results}) is a solution to the variational problem and 
also that the set of such probability measures is a singleton. In the finite-alphabet case
this implies, by a result in \cite{MR1085356}, the uniqueness of the equilibrium states. 
But it is not evident that the results from \cite{MR1085356} can be applied in cases where 
$\mathcal{A}$ is a general compact metric space. 
The main idea presented here to solve this problem is to 
use the Ruelle operator and DLR-Gibbs measures theory, suitably adapted for 
the one-dimensional one-sided lattice.  

In what follows we recall how the Ruelle operator associated to a H\"older potential 
and the equilibrium states are linked, in the finite-alphabet context. 
Next we move the discussion to more general settings.

For each $0<\alpha< 1$, we denote by $C^{\alpha}(\Omega)\equiv C^{\alpha}(\Omega,\mathbb{R})$ 
the space of all real valued $\alpha$-H\"older continuous functions which is 
defined as usual by 
\[
\left\{ 
f:\Omega\to\mathbb{R}: 
\mathrm{Hol}_{\alpha}(f)\equiv 
\sup_{ x\neq y}\frac{|f(x)-f(y)|}{(d_{\Omega}(x,y))^{\alpha}}<\infty
\right\}.
\]

When the potential $f$ has nice regularity 
properties such as $\alpha$-H\"older continuity, equilibrium states and fine 
properties of them can be obtained by using the 
Ruelle transfer operator $\mathscr{L}_{f}$, which is 
defined for each continuous function $\varphi$ by the following expression 
\[
\mathscr{L}_{f}(\varphi)(x) = \sum_{a\in\mathcal{A}} \exp(f(ax))\varphi(ax),
\qquad \text{where}\ \ ax\equiv (a,x_1,x_2,\ldots).
\]

If $f$ is an $\alpha$-H\"older continuous function 
then the Ruelle-Perron-Frobenius theorem 
ensures, among other things, that $\lambda_{f}$, 
the spectral radius of $\mathscr{L}_{f}$ 
acting on $C^{\alpha}(\Omega)$, 
is a positive maximal isolated eigenvalue and associated 
to it we have a strictly positive eigenfunction $h_f$.
Moreover, there is a Borel probability measure 
$\nu_{f}$ on $\Omega$,  such that 
$\mathscr{L}_{f}^{*}\nu_{f}=\lambda_{f}\nu_{f}$, 
where $\mathscr{L}_{f}^{*}$ is the Banach transpose 
of $\mathscr{L}_{f}:C(\Omega)\to C(\Omega)$, and
$C(\Omega)\equiv C(\Omega,\mathbb{R})$ is the space of all 
real continuous function from $\Omega$.
In this case there is a unique equilibrium state
for $f$ and it is given, up to a normalization, 
by the probability measure $\mu_f\equiv h_f\nu_{f}$, 
see \cite{MR1793194,MR1085356,MR0234697}. 
Besides solving the variational problem the spectral data 
of the Ruelle operator can also be used to obtain 
a variational formulation of the Kolmogorov-Sinai 
entropy of $\mu_{f}$ as follows 
\[
H(\mu_f) 
=
\log|\mathcal{A}|
+ 
\inf_{g\in C^{\alpha}(\Omega)}
\{ 
-\int_{\Omega} g\ d\mu_{f} +\log\lambda_{g}
\},
\]
see \cite{MR3377291}. 
If $f\equiv 0$, then $H(\mu_{f})=\log|\mathcal{A}|$ 
and therefore when the number of symbols in our alphabet
goes to infinity, i.e., $|\mathcal{A}|\to\infty$,
the Kolmogorov-Sinai entropy of $\mu_{f}$ will be infinity. 

To handle the case of infinite alphabets, 
a new definition for entropy of a Gibbs measure 
associated to a H\"older potential 
is proposed in \cite{MR3377291}. 
To circumvent the pointed-out problem with
Kolmogorov-Sinai entropy, the authors 
considered an {\it a priori} Borel probability measure $p$ on $\mathcal{A}$ having full support,
and next a $p$-dependent concept of entropy is given. 
It takes non-positive values and attain its supremum in the product measure
$\prod_{i\in\mathbb{N}} dp$. 
The introduction of an a priori 
measure is also a standard procedure, 
in Equilibrium Statistical Mechanics, 
when dealing with continuous spin systems, see \cite{MR1241537,MR2807681}.
However, in \cite{MR3377291} no mention to a possible relation between these two approaches is made.
In \cite{2015arXiv150801297G} the authors develop an abstract theory of Thermodynamic Formalism, 
where entropy is defined as a kind of Legendre-Fenchel dual of the topological pressure.
It worth to mention that similar results are well-known in the context of Statistical Mechanics, see for example \cite{MR517873}. 
Here we show that the specific entropy coincides
with the one introduced in \cite{MR3377291} and it can be extended to the 
classical Legendre-Fenchel transform of the topological pressure, 
thus providing a concrete representation (as a Thermodynamic Limit) 
of this function in cases where the alphabet $\mathcal{A}$ is a general compact metric space. 
Furthermore, we obtain (Theorem \ref{theoremdlr}) a version of 
the important identity (15.32) of \cite{MR2807681}
\[
\lim_{n\to\infty}\frac{1}{|\Lambda_n|}
\mathscr{H}_{\Lambda_n}(\mu|\nu)
=
P(\Phi)+\langle \mu,P\rangle -\mathtt{h}^{\mathtt{s}}(\mu)
\]
in the context of symbolic dynamics for potentials in the 
Walters class, see Definition \ref{def-walters-class}. 
We also show that the entropy 
defined in \cite{MR3377291} is equal to  
the specific entropy (or mean entropy or entropy rate) 
commonly used in Statistical Mechanics (see \cite{MR2807681}) and
as by-product a variational formulation for the specific entropy is derived.
Afterwards, these results are applied to prove uniqueness of equilibrium states for 
potentials in Walters class, when $\mathcal{A}$ is a general compact metric alphabet.

We should mention that our results about the uniqueness of the equilibrium states 
for potentials in Walters space can not be deduced from the recent results in \cite{2015arXiv150801297G}. 
The first reason is related to the finite-to-one map hypothesis imposed there,
which is broken when the alphabet is infinite. 
The second reason is that the Ruelle operator associated to a Walters potential does not have, 
in general, the spectral gap property, see \cite{MR3538412}. In such cases the arguments given in 
\cite{2015arXiv150801297G} can not be used.
Although the references \cite{CL-rcontinuas-2016,MR3538412,MR3377291} already considered compact
alphabets and equilibrium states in such context, the uniqueness results discussed in these 
papers concern to the eigenmeasures of the Banach transpose of the Ruelle operator. 
As mentioned earlier, in the finite-alphabet case, the uniqueness of the eigenmeasures implies 
the uniqueness of the equilibrium states for H\"older and Walters potentials. 
For general compact metric alphabets more work is required. 
The additional results needed to prove the uniqueness of the 
equilibrium states are obtained in Section \ref{sec-main-results}.

In \cite{CL-rcontinuas-2016,MR3538412,MR3377291} a 
concept of entropy is considered 
in the context of compact metric alphabets. It is defined by a variational principle, 
but it was not recognized as the negative of the Legendre-Fenchel transform of the 
pressure functional. 
Here we obtain as a byproduct of the proof of Theorem \ref{teo-equiv-entrop}:  
the entropy considered in these works
is actually the negative of the Legendre-Fenchel dual of the pressure functional and, moreover, is equal to
the specific entropy. This equality, the content of Theorem \ref{teo-equiv-entrop}, has a further consequence
due to Proposition 15.14 in \cite{MR2807681} which reads: 
the entropy considered in these works is 
an affine function, when restricted to the subspace of shift-invariant Borel probability 
measures. Although this was a known result for dynamical systems with finite topological entropy,  
it is new for uncountable alphabets.

\section{Specific Entropy}\label{preliminar}

\hspace{.5cm} In this section we recall the definition and basic properties of the 
specific entropy, the exposition follows closely the reference \cite{MR2807681}.

From now on $\mathcal{A}$ is a compact metric 
space and $\mathscr{F}$ denotes the 
sigma-algebra of $\Omega$ generated by the open sets. 
Let $\mu, \nu$ be two probability measures on the 
measurable space $(\Omega,\mathscr{F})$. 
The negative of the \textit{relative entropy} of $\mu$ with respect to 
$\nu$ on the sub-sigma-algebra $\mathscr{A}$,
notation $\mathscr{H}_{\mathscr{A}}(\mu|\nu)$, is the extended real number
\[
\mathscr{H}_{\mathscr{A}}(\mu|\nu)
\equiv 
\begin{cases}
\displaystyle\int_{\Omega} 
\frac{d\mu|_{\mathscr{A}}}{d\nu|_{\mathscr{A}}} 
\log  \left(\frac{d\mu|_{\mathscr{A}}}{d\nu|_{\mathscr{A}}} \right)
\, d\nu, &\ \text{if}\ \mu\ll\nu \ \text{on}\ \mathscr{A};
\\[0.5cm]
\infty,&\ \text{otherwise}.
\end{cases}
\]

For each finite $\Lambda \subset \mathbb{N}$ consider the projection 
$\pi_{\Lambda}:\Omega\to \mathcal{A}^{\Lambda}$ given by 
$\pi_{\Lambda}(x) = (x_i)_{i\in \Lambda}$.
We denote by $\mathscr{F}_{\Lambda}$ the sigma-algebra generated by 
the projections $\{\pi_{\Gamma}: \Gamma\subset \Lambda \}$ and 
we define the relative entropy of $\mu$ with respect to $\nu$ in $\mathscr{F}_{\Lambda}$
by 
$
\mathscr{H}_{\Lambda}\left( \mu | \nu\right)
\equiv 
\mathscr{H}_{\mathscr{F}_{\Lambda}}\left( \mu | \nu\right).
$ 

Fix an a priori probability measure $p$ on $\mathcal{A}$.  
For each probability measure $\mu$ on  $\Omega$ and a finite volume 
$\Lambda \subset\mathbb{N}$, 
the relative entropy of $\mu$ in $\Lambda$, with respect to $p$, is defined by
$
\mathscr{H}_{\Lambda}(\mu) 
\equiv 
- 
\mathscr{H}_{\Lambda}
(\mu\, \big|\  \textstyle \prod_{i\in\mathbb{N}}p ).
$ 
If $\Lambda_n\equiv \{1,\ldots,n \}$ and $\mu\in\mathscr{M}_{\sigma}(\Omega)$, 
is shown in \cite{MR2807681} that the following limit 
\[
\mathtt{h}^{\mathtt{s}}(\mu) 
= 
\lim_{n \rightarrow \infty} \frac{\mathscr{H}_{{\Lambda}_n}(\mu)}{n}
\]
always exists in $[-\infty,0]$ and $\mathtt{h}^{\mathtt{s}}(\mu)$ 
is called the {\bf specific entropy} per site of $\mu$ relative to the a priori 
probability measure $p$. The specific entropy $\mathtt{h}^{\mathtt{s}}$ is always
a concave and upper semicontinuous function. 
Since $\mathtt{h}^{\mathtt{s}}(\prod_{i\in\mathbb{N}}p)=0$, for any choice $p$, 
$\mathtt{h}^{\mathtt{s}}$ is not identically constant equal to $-\infty$. 
An extreme case occurs when $p$ is a Dirac measure concentrated on an arbitrary point of $\mathcal{A}$. 
In this case the product measure $\prod_{i\in\mathbb{N}}p$ is the only $\sigma$-invariant 
probability measure for which this function take finite values. For the proof of these
properties and more details on specific entropy, see \cite{MR2807681}.

\section{Main Results}\label{sec-main-results}

\hspace{.5cm} In this section we obtain a variational formulation for the specific entropy 
and also prove the uniqueness of equilibrium states for a large class of potentials.
For the sake of simplicity, we present the argument for H\"older potentials and 
point out, in last section, what are the needed changes to prove the theorem 
in more general cases.

Let $(\mathcal{A},d_{\mathcal{A}})$ be an arbitrary compact metric space and
consider the symbolic space $\Omega = \mathcal{A}^{\mathbb{N}}$ 
equipped with a metric $d_{\Omega}$ which induces the product topology. 
For example, 
$
d_{\Omega}(x,y) 
\equiv \sum_{n=1}^{\infty} 2^{-n} 
d_{\mathcal{A}}(x_n,y_n)/(1+d_{\mathcal{A}}(x_n,y_n)).
$ 

Given an a priori measure $p$ on $\mathcal{A}$ and an $\alpha$-H\"older 
continuous potential $f$ we define the Ruelle operator 
$\mathscr{L}_{f}:C^{\alpha}(\Omega)\to C^{\alpha}(\Omega)$ 
as being the linear operator sending $\varphi$ to $\mathscr{L}_{f}(\varphi)$, 
which is given by the following expression 
\begin{align*}
\mathscr{L}_{f}(\varphi)(x) 
= 
\int_{\mathcal{A}} \exp(f(ax))\varphi(ax)\, dp(a),
\qquad
\text{where}\ \ ax\equiv (a,x_1,x_2,\ldots).
\end{align*}

A potential $f$ is said to be normalized if $\mathscr{L}_{f}(1)(x)=1$ for all $x\in\Omega$. 
By using the generalization of Ruelle-Perron-Frobenius theorem provided in \cite{MR3377291}
we can associate to any H\"older potential $f$ a cohomologous normalized potential 
$\bar{f}$ given by
\begin{align}\label{f-bar}
\bar f
=
f + \log h_{f} - \log (h_{f} \circ \sigma)  - \log \lambda_{f},
\end{align}
where $\lambda_f$ is a maximal eigenvalue of $\mathscr{L}_{f}$ and 
$h_{f}$ is a strictly positive $\alpha$-H\"older eigenfunction associated
$\lambda_f$. The authors also proved that  
$
\mathcal{G}^{*}(\bar{f}) 
\equiv  
\{
\nu \in \mathscr{M}_1(\Omega): \mathscr{L}_{\bar{f}}^*\nu =\nu 
\}
$
is a singleton and contained in $\mathscr{M}_{\sigma}(\Omega)$.
This measure is called here the Gibbs measure associated to the 
potential $f$.

Now we consider the following family of probability kernels
$(\gamma_n)_{n\geq 1}$, where for each $n\geq 1$ the kernel 
$\gamma_n:\mathscr{F}\times \Omega\to [0,1]$ is given by 
$
\gamma_n(A|y)
\equiv 
\mathscr{L}^n_{f}(1_A)(\sigma^n(y)),
$
where $f\in C^{\alpha}(\Omega)$ is a normalized potential.
For any probability measure $\nu$, on $\Omega$, and $y\in\Omega$ 
we define a probability measure
$\nu\gamma_n(\cdot|y)\equiv \int_{\Omega}\gamma_n(\cdot|y) \, d\nu(y)$.
We remark that if $\nu$ is such that 
$\mathscr{L}_{f}^{*}\nu=\nu$ then for any $A\in\mathscr{F}$ and 
$y\in\Omega$ we have 
\begin{align*}
\nu\gamma_n(A|y)
&=
\int_{\Omega} \mathscr{L}^n_{f}(1_A)\circ\sigma^n(y) \, d\nu(y)
=
\int_{\Omega} \mathscr{L}^n_{f}(1_A)(y) \, d\nu(y)
\\
&=
\int_{\Omega} 1_A(y) \, d[(\mathscr{L}^{*}_{f})^{n}\nu](y)
=\nu(A).
\end{align*}

Our next theorem is inspired by Theorem 15.30  of \cite{MR2807681}. 
The hypotheses there are not fully satisfied in  
our setting. Firstly, here we are working on one-dimensional one-sided lattice. 
Secondly, the family of kernels considered here is not defined by  
a uniformly summable translation invariant interaction. 
The main difference of our proof is the use of duality properties of the Ruelle operator. 

\begin{theorem}\label{theoremdlr} 
	Let $p$ be an a priori probability measure on $\mathcal{A}$ having full
	support and $f \in C^{\alpha}(\Omega)$.
	Then for each $\mu \in \mathscr{M}_{\sigma}(\Omega)$ and 
	$ \nu \in \mathcal{G}^{*}(\bar{f})$, the  following limit exists
	\[
	\lim_{n \rightarrow \infty} 
	\frac{1}{n} \mathscr{H}_{\Lambda_n}(\mu | \nu)
	\equiv 
	\mathtt{h}(\mu | \nu)
	=
	\log\lambda_{f} - \int_{\Omega}f \, d\mu - \mathtt{h}^{\mathtt{s}}(\mu).
	\]
\end{theorem}

\begin{proof}
	If for some $n\in\mathbb{N}$, we have $\mathscr{H}_{\Lambda_n}(\mu | \nu)=+\infty$, then 
	$\mathscr{H}_{\Lambda_n}(\mu |\nu\gamma_n(\cdot|y))=+\infty$ and therefore 
	$\mathscr{H}_{\Lambda_n}(\mu |\gamma_n(\cdot|y))=+\infty$, so 
	$\mu|_{\mathscr{F}_{\Lambda_n}}$ is not absolutely 
	continuous with respect to  $\gamma_n(\cdot|y)|_{\mathscr{F}_{\Lambda_n}}$. Since $f$ is bounded, it follows from the definition of 
	$\gamma_n(\cdot|y)$ that $\mu|_{\mathscr{F}_{\Lambda_n}}$ is not absolutely 
	continuous with respect to $\pmb{p}|_{\mathscr{F}_{\Lambda_n}}$. 
	Since the relative entropy is an increasing function of $\mathscr{F}_{\Lambda_n}$
	we have 
	$
	\mathscr{H}_{\Lambda_j}(\mu|\gamma_j(\cdot|y)) = + \infty, \ \forall \  j \geq n
	$
	and also $\mathtt{h}^{\mathtt{s}}(\mu)=-\infty$, which proves the theorem in this case. 
	Therefore we can assume $\mathscr{H}_{\Lambda_n}(\mu|\gamma_n(\cdot|y)) < + \infty $, for all $n \in \mathbb{N}$.
	To lighten the notation 
	we write $d\pmb{p}\equiv \prod_{i\in \mathbb{N}}dp$
	to denote the product measure. Let $\bar{f}$ a normalized potential 
	cohomologous to $f$ and 
	$
	\gamma_n(A|y)
	\equiv 
	\mathscr{L}^n_{\bar{f}}(1_A)(\sigma^n(y)).
	$
	If $\mathscr{H}_{\Lambda_n}(\mu|\gamma_n(\cdot|y))$ is finite for all $n\geq 1$ 
	then   
	\begin{align*}
	\mathscr{H}_{\Lambda_n}(\mu|\gamma_n&(\cdot|y))
	=
	\int_{\Omega} 
	\frac{ d\mu|_{ \mathscr{F}_{\Lambda_{n}}  }  }
	{d{\gamma_n(\cdot|y)}|_{\mathscr{F}_{\Lambda_n}} }
	\log 
	\frac{ d\mu|_{ \mathscr{F}_{\Lambda_{n}}  }  }
	{d{\gamma_n(\cdot|y)}|_{\mathscr{F}_{\Lambda_n}} }
	\, 
	d\gamma_{n}(\cdot|y)
	\\
	&=
	\mathscr{L}^n_{\bar{f}}
	\left(  
	\frac{ d\mu|_{ \mathscr{F}_{\Lambda_{n}}  }  }
	{d{\gamma_n(\cdot|y)}|_{\mathscr{F}_{\Lambda_n}} }
	\log 
	\frac{ d\mu|_{ \mathscr{F}_{\Lambda_{n}}  }  }
	{d{\gamma_n(\cdot|y)}|_{\mathscr{F}_{\Lambda_n}} }
	\right)
	(\sigma^n(y))
	\\
	&=
	\int_{\Omega}
	\exp(S_n(\bar{f}))
	\frac{ d\mu|_{ \mathscr{F}_{\Lambda_{n}}  }  }
	{d{\gamma_n(\cdot|y)}|_{\mathscr{F}_{\Lambda_n}} }
	\log 
	\frac{ d\mu|_{ \mathscr{F}_{\Lambda_{n}}  }  }
	{d{\gamma_n(\cdot|y)}|_{\mathscr{F}_{\Lambda_n}} }
	\prod_{i\in\Lambda_n}dp\times \prod_{i\in\Lambda_n^c}d\delta_{y_i}
	\\
	&=
	\int_{\Omega}
	\frac{d{\gamma_n(\cdot|y)}|_{\mathscr{F}_{\Lambda_n}} }
	{d\pmb{p}|_{ \mathscr{F}_{\Lambda_{n}}  }    }
	\frac{ d\mu|_{ \mathscr{F}_{\Lambda_{n}}  }  }
	{d{\gamma_n(\cdot|y)}|_{\mathscr{F}_{\Lambda_n}} }
	\log 
	\frac{ d\mu|_{ \mathscr{F}_{\Lambda_{n}}  }  }
	{d{\gamma_n(\cdot|y)}|_{\mathscr{F}_{\Lambda_n}} }
	\prod_{i\in\Lambda_n}dp\times \prod_{i\in\Lambda_n^c}d\delta_{y_i}.
	\\
	&=
	\int_{\Omega}
	\frac{ d\mu|_{ \mathscr{F}_{\Lambda_{n}}  }  }
	{ d\pmb{p}|_{ \mathscr{F}_{\Lambda_{n}}  }    }
	\log 
	\frac{ d\mu|_{ \mathscr{F}_{\Lambda_{n}}  }  }
	{d{\gamma_n(\cdot|y)}|_{\mathscr{F}_{\Lambda_n}} }
	\prod_{i\in\Lambda_n}dp\times \prod_{i\in\Lambda_n^c}d\delta_{y_i},
	\end{align*}
	where $S_n(\bar{f})\equiv \bar{f}+\bar{f}\circ\sigma+\ldots+\bar{f}\circ\sigma^{n-1}$.
	Since the above integrand is $\mathscr{F}_{\Lambda_n}$-measurable 
	we get 
	\[
	\mathscr{H}_{\Lambda_n}(\mu|\gamma_n(\cdot|y))
	=
	\int_{\Omega}
	\frac{ d\mu|_{ \mathscr{F}_{\Lambda_{n}}  }  }
	{ d\pmb{p}|_{ \mathscr{F}_{\Lambda_{n}}  }    }
	\log 
	\frac{ d\mu|_{ \mathscr{F}_{\Lambda_{n}}  }  }
	{d{\gamma_n(\cdot|y)}|_{\mathscr{F}_{\Lambda_n}} }
	\ d\pmb{p}.
	\]
	From the properties of the Radon-Nikodym derivative 
	follows  that 
	\begin{align*}
	\mathscr{H}_{\Lambda_n}(\mu|\gamma_n(\cdot|y))
	&=
	\int_{\Omega}
	\frac{ d\mu|_{ \mathscr{F}_{\Lambda_{n}}  }  }
	{d\pmb{p}|_{ \mathscr{F}_{\Lambda_{n}}  } }
	\log 
	\left(
	\frac{ d\mu|_{ \mathscr{F}_{\Lambda_{n}}  }  }
	{d{\gamma_n(\cdot|y)}|_{\mathscr{F}_{\Lambda_n}} }
	\frac{ d{\gamma_n(\cdot|y)}|_{\mathscr{F}_{\Lambda_n}}  }
	{ d\pmb{p}|_{ \mathscr{F}_{\Lambda_{n}}  } }
	\right) 
	\ d\pmb{p}
	\\
	&\qquad -
	\int_{\Omega}
	\frac{ d\mu|_{ \mathscr{F}_{\Lambda_{n}}  }  }
	{d\pmb{p}|_{ \mathscr{F}_{\Lambda_{n}}  } }
	\log 
	\frac{ d{\gamma_n(\cdot|y)}|_{\mathscr{F}_{\Lambda_n}}  }
	{ d\pmb{p}|_{ \mathscr{F}_{\Lambda_{n}}  } }
	\ d\pmb{p}
	\\[0.3cm]
	&=
	\int_{\Omega}
	\frac{ d\mu|_{ \mathscr{F}_{\Lambda_{n}}  }  }
	{d\pmb{p}|_{ \mathscr{F}_{\Lambda_{n}}  } }
	\log 
	\frac{ d\mu|_{ \mathscr{F}_{\Lambda_{n}}  }  }
	{d\pmb{p}|_{ \mathscr{F}_{\Lambda_{n}}  } }
	\ d\pmb{p}
	-
	\int_{\Omega} S_n(\bar{f})(x_{\Lambda_n}y_{\Lambda_{n}^c})  d\mu(x) 
	\\
	&=
	-\mathscr{H}_{\Lambda_n}(\mu)
	-\int_{\Omega} S_n(\bar{f})(x_{\Lambda_n}y_{\Lambda_{n}^c})  d\mu(x),
	\end{align*}
	where
	$(x_{\Lambda}y_{\Lambda^c})_i=x_i$, if $i\in \Lambda$ and 
	$(x_{\Lambda}y_{\Lambda^c})_i=y_i$, otherwise.
	
	\medskip 
	Note that 
	\begin{align}\label{eq-free-energy}
	\int_{\Omega} \bar{f}\, d\mu 
	= 
	\lim_{n \rightarrow \infty} 
	\frac{1}{n}
	\int_{\Omega} S_n(\bar{f})(x_{\Lambda_n}y_{\Lambda_n^c}) \, d\mu(x).
	\end{align}
	Indeed, for all $n\geq 1$, we have 
	$
	|S_n(\bar{f})(x_{\Lambda_n}y_{\Lambda_n^c})- S_n(\bar{f})(x)|
	\leq 
	\sum_{j=0}^n 2^{-\alpha j} \mathrm{Hol}_{\alpha}(\bar{f}).
	$
	From this observation and the $\sigma$-invariance of $\mu$ 
	the claim follows.
	\bigskip 
	
	Recall that for any $\mu\in \mathscr{M}_{\sigma}(\Omega)$ we have 
	$\mathscr{H}_{\Lambda_n}(\mu)/n\to \mathtt{h}^{\mathtt{s}}(\mu)$,
	when $n\to\infty$. This convergence together with \eqref{eq-free-energy}
	implies the existence of the following limit  
	\[
	\lim_{n\to\infty}
	\frac{\mathscr{H}_{\Lambda_n}(\mu|\gamma_n(\cdot|y))}{n}
	=
	-\mathtt{h}^{\mathtt{s}}(\mu)-\int_{\Omega}\bar{f}\, d\mu
	=
	\log \lambda_{f}-\mathtt{h}^{\mathtt{s}}(\mu)-\int_{\Omega}f\, d\mu,
	\]
	where in the last equality we used that $\mu\in\mathscr{M}_{\sigma}(\Omega)$
	and the expression \eqref{f-bar}.

	\medskip
	To finish the proof we only need to show that   
	$n^{-1}\mathscr{H}_{\Lambda_n}(\mu|\gamma_n(\cdot|y))\to \mathtt{h}(\mu | \nu)$,
	when $n\to\infty$, for any choice of $y\in\Omega$.
	\begin{align}
	\mathscr{H}_{\Lambda_n} (\mu  &| \nu  )
	=
	\mathscr{H}_{\Lambda_n} (\mu  | \nu\gamma_n(\cdot| y)  )
	\nonumber
	\\
	&=
	\int_{\Omega}        
	\frac{d\mu|_{\mathscr{F}_{\Lambda_n} }  }{ d [{\nu\gamma_n(\cdot| y)}|_{\mathscr{F}_{\Lambda_n} }]}
	\log 
	\frac{d\mu|_{\mathscr{F}_{\Lambda_n} }  }{ d [{\nu\gamma_n(\cdot| y)}|_{\mathscr{F}_{\Lambda_n} }]}
	d [\nu\gamma_n(\cdot| y)]
	\nonumber
	\\[0.3cm]
	&=
	\int_{\Omega}        
	\frac{d\mu|_{\mathscr{F}_{\Lambda_n} }  }{ d\pmb{p}|_{\mathscr{F}_{\Lambda_n} } }
	\left(
	\frac{d [{\nu\gamma_n(\cdot| y)}|_{\mathscr{F}_{\Lambda_n} }]  }{ d\pmb{p}|_{\mathscr{F}_{\Lambda_n} } }
	\right)^{-1}
	\log 
	\frac{d\mu|_{\mathscr{F}_{\Lambda_n} }  }{ d [{\nu\gamma_n(\cdot| y)}|_{\mathscr{F}_{\Lambda_n} }]}
	d [\nu\gamma_n(\cdot| y)]
	\nonumber 
	\\[0.3cm]
	&=
	\int_{\Omega}        
	\log 
	\frac{d\mu|_{\mathscr{F}_{\Lambda_n} }  }{ d [{\nu\gamma_n(\cdot| y)}|_{\mathscr{F}_{\Lambda_n} }]}
	\frac{d [{\gamma_n(\cdot| y)}|_{\mathscr{F}_{\Lambda_n} }]  }{ d\pmb{p}|_{\mathscr{F}_{\Lambda_n} } }
	d \mu
	-
	\int_{\Omega}        
	\log 
	\frac{d [{\gamma_n(\cdot| y)}|_{\mathscr{F}_{\Lambda_n} }]  }{ d\pmb{p}|_{\mathscr{F}_{\Lambda_n} } }
	d \mu
	\nonumber
	\\[0.3cm]
	&=
	\int_{\Omega}        
	\log 
	\frac{d\mu|_{\mathscr{F}_{\Lambda_n} }  }{ d [{\gamma_n(\cdot| y)}|_{\mathscr{F}_{\Lambda_n} }]}
	d \mu
	+
	\int_{\Omega}        
	\log 
	\frac{d [{\gamma_n(\cdot| y)}|_{\mathscr{F}_{\Lambda_n} }]  }{ d\pmb{p}|_{\mathscr{F}_{\Lambda_n} } }
	\left(
	\frac{d [{\nu\gamma_n(\cdot| y)}|_{\mathscr{F}_{\Lambda_n} }]  }{ d\pmb{p}|_{\mathscr{F}_{\Lambda_n} } }
	\right)^{-1}
	d \mu
	\nonumber 
	\\[0.3cm]
	&=
	\mathscr{H}_{ \Lambda_n }(\mu | \gamma_n(\cdot|y)  )
	\nonumber
	\\
	&\qquad +
	\int_{\Omega}
	\left[
	S_n(f)(x_{\Lambda_n}y_{\Lambda_n^c})
	-
	\log 
	\int_{\Omega} 
	\exp(S_n(f)(x_{\Lambda_n}z_{\Lambda_n^c}))
	\, d\nu(z) 	\
	\right]
	d\mu(x)
	\label{eq-aux1}
	\\[0.3cm]
	&=
	\mathscr{H}_{ \Lambda_n }(\mu | \gamma_n(\cdot|y)  )
	+o(n)\nonumber,
	\end{align}
	where the last expression follows from \eqref{eq-free-energy} together with the
	inequality
	$
	|\log \int_{\Omega}\exp(\varphi) d\mu - \log \int_{\Omega}\exp(\psi) d\mu|\leq \|\varphi-\psi\|_{\infty}.
	$
\end{proof}

\begin{corollary}\label{corineq} 
	For all $f \in C^{\alpha}(\Omega)$ and $ \mu \in \mathscr{M}_{\sigma}(\Omega)$ 
	we have $\mathtt{h}(\mu | \mu_{\bar{f}}) \geq 0$. 
\end{corollary}

\bigskip 

We now present the entropy considered in \cite{MR3377291} and then 
we prove that it coincides with the specific entropy on the set
of all Borel shift-invariant probability measures.

\begin{definition}
	Given a Borel probability measure $\mu$ on $\Omega$, we define its entropy as
	follows
	\begin{align*}
	\mathtt{h}^{\!\mathtt{v}}(\mu)
	\equiv 
	\mathtt{h}^{\!\mathtt{v},p}(\mu)
	\equiv
	\inf_{g \in C^{\alpha}(\Omega)}
	\left\lbrace - \int_{\Omega} g \  d\mu + \log \lambda_g \right\rbrace ,
	\end{align*}
	where $\lambda_g$ is the maximal eigenvalue of $\mathscr{L}_g$. 
\end{definition}

\begin{theorem}\label{teo-equiv-entrop}
	For all $\mu\in \mathscr{M}_{\sigma}(\Omega)$  we have  
	$\mathtt{h}^{\mathtt{s}}(\mu) = \mathtt{h}^{\!\mathtt{v}}(\mu)$.
\end{theorem}

\begin{proof}
	Given $\mu\in \mathscr{M}_{\sigma}(\Omega)$
	it follows from Corollary \ref{corineq} that 
	$\mathtt{h}^{\mathtt{s}}(\mu)\leq \log \lambda_{g}- \int_{\Omega} g\, d\mu $,
	for all $g\in C^{\alpha}(\Omega)$. 
	Therefore $\mathtt{h}^{\mathtt{s}}(\mu)\leq \mathtt{h}^{\!\mathtt{v}}(\mu)$.
	
	The remainder of the proof is by contradiction.
	For the sake of simplicity, we will work on $\mathscr{M}_{s}(\Omega)$,
	the topological vector space of all finite Borel signed measures 
	endowed with the weak-$*$ topology, and with 
	extensions $\mathtt{h}^{\mathtt{s}},\mathtt{h}^{\!\mathtt{v}}:\mathscr{M}_{s}(\Omega)\to [-\infty,0]$ 
	of $\mathtt{h}^{\mathtt{s}}$ and $\mathtt{h}^{\!\mathtt{v}}$, respectively, given by
	\begin{align*}
	\mathtt{h}^{\mathtt{s}}(\mu)
	=
	\begin{cases}
	\mathtt{h}^{\mathtt{s}}(\mu),& 
	\text{if}\ \mu\in\mathscr{M}_{\sigma}(\Omega)
	\\
	-\infty,&\text{otherwise}.
	\end{cases}
	\qquad 
	\mathtt{h}^{\!\mathtt{v}}(\mu)
	=
	\begin{cases}
	\mathtt{h}^{\!\mathtt{v}}(\mu),& 
	\text{if}\ \mu\in\mathscr{M}_{\sigma}(\Omega)
	\\
	-\infty,&\text{otherwise}.
	\end{cases}
	\end{align*}
	
	We split the proof in two separate cases.
	\\
	\noindent{\bf Case 1}. 
	Suppose that there is $\nu\in \mathscr{M}_{s}(\Omega)$ such that 
	$-\mathtt{h}^{\!\mathtt{v}}(\nu)<-\mathtt{h}^{\mathtt{s}}(\nu)< +\infty$.
	Since the negative of the extension of the specific entropy 
	$-\mathtt{h}^{\mathtt{s}}:\mathscr{M}_{s}(\Omega)\to [0,+\infty]$ 
	is convex, lower semicontinuous and bounded from below, we have that its epigraph 
	$
	\mathrm{epi}(-\mathtt{h}^{\mathtt{s}})
	\equiv 
	\{ (\mu,t) \in \mathscr{M}_{\sigma}(\Omega)\times\mathbb{R} : 
	-\mathtt{h}^{\mathtt{s}}(\mu)\leq t
	\}
	$
	is a convex closed subset of $\mathscr{M}_{s}(\Omega)\times \mathbb{R}$.
	By assumption we have 
	$(\nu,-\mathtt{h}^{\!\mathtt{v}}(\nu))\notin \mathrm{epi}(-\mathtt{h}^{\mathtt{s}})$.
	Therefore, there are  $c\in\mathbb{R}$ and a continuous linear functional 
	$F:\mathscr{M}_{s}(\Omega)\times\mathbb{R}\to\mathbb{R}$, 
	such that for all $(\mu,t)\in \mathrm{epi}(-\mathtt{h}^{\mathtt{s}})$
	we have $F(\nu,-\mathtt{h}^{\!\mathtt{v}}(\nu))<c<F((\mu,t))$. Since  
	$\mathscr{M}_{s}(\Omega)$ endowed with the weak-$*$ topology,  the functional
	$F$ can be represented as follows $F((\mu,t))=\int_{\Omega} \varphi\, d\mu+at$,
	for some $\varphi\in C(\Omega)$ and $a\in\mathbb{R}$.
	From the previous inequality follows that 
	\[
	\int_{\Omega} \varphi\, d\nu-a\mathtt{h}^{\!\mathtt{v}}(\nu)
	< 
	c
	< 
	\int_{\Omega} \varphi\, d\mu+at, 
	\quad
	\text{for all}\ (\mu,t)\in\mathrm{epi}(-\mathtt{h}^{\mathtt{s}})
	\] 
	Since $\mathscr{M}_{\sigma}(\Omega)$ is compact, 
	and $C^{\alpha}(\Omega)$ is dense in $C(\Omega)$,
	with respect to the uniform norm, up to small change in $c$ we can 
	assume that $\varphi$ in above inequality is a H\"older continuous function.
	From this inequality it is easy to deduce that $a>0$. Without
	loss of generality we can assume that $a=1$. 
	From Corollary \ref{corineq}, definition of $\mathrm{epi}(-\mathtt{h}^{\mathtt{s}})$, 
	the above inequality, and Theorem 3 of \cite{MR3377291}  we have 
	\begin{align*}
	\log \lambda_{(-\varphi)}
	=
	\sup_{\mu\in\mathscr{M}_{\sigma}(\Omega)} 
	\{\int_{\Omega} (-\varphi)\, d\mu + \mathtt{h}^{\mathtt{s}}(\mu)\}
	\leq 
	-c
	<
	\int_{\Omega} (-\varphi)\, d\nu+\mathtt{h}^{\!\mathtt{v}}(\nu)
	\leq 
	\log \lambda_{(-\varphi)}
	\end{align*}
	which is a contradiction. 
	Therefore, for all $\mu\in\mathscr{M}_{\sigma}(\Omega)$ such that 
	$\mathtt{h}^{\mathtt{s}}(\mu)>-\infty$ we have 
	$\mathtt{h}^{\mathtt{s}}(\nu)=\mathtt{h}^{\!\mathtt{v}}(\nu)$.
	
	\noindent{\bf Case 2}.
	Now we have to prove that for all $\mu\in\mathscr{M}_{\sigma}(\Omega)$
	such that $-\mathtt{h}^{\mathtt{s}}(\mu)= +\infty$, we have
	$-\mathtt{h}^{\!\mathtt{v}}(\mu)=+\infty$.  
	The idea is to prove that we can reduce this to the previous case. 
	Suppose that for some 
	$\nu\in\mathscr{M}_{\sigma}(\Omega)$ we have 
	$-\mathtt{h}^{\!\mathtt{v}}(\nu)<\mathtt{h}^{\mathtt{s}}(\nu)= +\infty$.
	Since $-\mathtt{h}^{\mathtt{s}}$ is a convex and lower semicontinuous function from 
	$\mathscr{M}_{s}(\Omega)$ to $[0,+\infty]$, then it is pointwise supremum
	of a family $\mathfrak{F}$ of continuous affine functions from $\mathscr{M}_{s}(\Omega)$
	to $\mathbb{R}$, see Proposition 3.1 of \cite{MR0463994}.  
	A generic member of $\mathfrak{F}$ is a function of the type 
	$\mu\longmapsto \xi(\mu)+ C$, where $\xi\in \mathscr{M}_{s}(\Omega)^*$
	and as observed before can be represented as $\mu\longmapsto \int_{\Omega} g\,d\mu+ C$,
	for some $g\in C(\Omega)$.
	Let $\mathfrak{D}$ denote the family of all affine functions of the form 
	$\mathscr{M}_{s}(\Omega)\ni \mu\longmapsto \int_{\Omega} g\, d\mu - \log \lambda_{g}$,
	with $g$ varying in $C^{\alpha}(\Omega)$.
	Since for all $\mu\in \mathscr{M}_{s}(\Omega)$, we have 
	$\mathtt{h}^{\mathtt{s}}(\mu)\leq \mathtt{h}^{\!\mathtt{v}}(\mu)$
	then we can assume that $\mathfrak{D}\subsetneq \mathfrak{F}$. 
	Given $M>-\mathtt{h}^{\!\mathtt{v}}(\nu)$ there is 
	$F\in\mathfrak{F}\setminus\mathfrak{D}$ such that $F(\nu)>M$. On the other hand,
	we have 
	$\mathtt{h}^{\mathtt{s}}(\prod_{i\in\mathbb{N}}p)=0$
	so $F(\prod_{i\in\mathbb{N}}p)\leq 0$. Without loss of generality we can 
	assume that the last inequality is strict, 
	and for all $\mu\in\mathscr{M}_{s}(\Omega)$, 
	we have $F(\mu) = \int_{\Omega} \varphi d\mu +at$, for some 
	$\varphi\in C^{\alpha}(\Omega)$ and $a\in\mathbb{R}$. 
	Note that $F$ is everywhere less than 
	$-\mathtt{h}^{\mathtt{s}}$ and separates $\mathrm{epi}(-\mathtt{h}^{\mathtt{s}})$ from 
	$(\nu,-\mathtt{h}^{\!\mathtt{v}}(\nu))$. Therefore $a\neq 0$ and 
	the proof is finished by proceeding as in Case 1.  
\end{proof}

\begin{theorem}\label{teo-uniqueness-EqStates-HolderCase}
	If $f \in C^{\alpha}(\Omega)$ then 
	$\mathtt{h}(\mu |\mu_{\bar{f}}) = 0$ if and only if 
	$\mu \in \mathcal{G}^{*}(\bar{f})$. In particular,
	the set of the equilibrium states for $f$ is a singleton. 
\end{theorem}
\begin{proof}
	If $f \in C^{\alpha}(\Omega)$ then 
	$\mathcal{G}^{*}(\bar{f})$ is a singleton, see \cite{MR3377291}, 
	and so $\mathtt{h}(\mu |\mu_{\bar{f}}) = 0$. 
	
	Conversely, suppose that $\mu\in\mathscr{M}_{\sigma}(\Omega)$ 
	is such that $\mathtt{h}(\mu |\mu_{\bar{f}}) = 0$.
	To prove that $\mu\in\mathcal{G}^{*}(\bar{f})$ it is enough  
	to show that for each fixed $n_0\in\mathbb{N}$ we have 
	$\mu\gamma_{n_0} = \mu$, see reference \cite{CL-rcontinuas-2016} 
	Definition 4 and remark b) below it, and Theorem 2. 
	Since we are assuming that $\mathtt{h}(\mu |\mu_{\bar{f}}) = 0$,
	and we know that 
	$\Lambda\mapsto \mathscr{H}_{ \Lambda}(\mu | \mu_{\bar{f}} )$
	is a non-decreasing function (see \cite[p. 310]{MR2807681}) it follows that 
	$\mathscr{H}_{ \Lambda}(\mu | \mu_{\bar{f}} )<\infty$, for any finite 
	$\Lambda\subset\mathbb{N}$. 
	
	As a warm-up, let us first prove the theorem in the case 
	$
	\sup\{ 
	\mathscr{H}_{ \Lambda_n}(\mu | \mu_{\bar{f}} )\ : n\in\mathbb{N} 
	\}
	<
	+\infty
	$. Under this assumption we have
	$\mathscr{H}_{ \Lambda_n}(\mu | \mu_{\bar{f}} )<+\infty$, 
	for all $n\in\mathbb{N}$, 
	and so there exists the Radon-Nikodym derivative 
	\begin{align}\label{def-phin}
	\varphi_{n}
	\equiv 
	\frac{d\mu|_{\mathscr{F}_{\Lambda_n}}}{d\mu_{\bar{f}}|_{\mathscr{F}_{\Lambda_n}}}.
	\end{align}
	Recall that $\varphi_{n}\geq 0$ and $\mathscr{F}_{\Lambda_n}$-measurable.
	It is a simple matter to check that $(\varphi_{n})_{n\in\mathbb{N}}$ is 
	a martingale relative to $\mu_{\bar{f}}$.
	Indeed, since $\mathscr{F}_{\Lambda_n}$ defines an 
	increasing filtration, for every $m > n$ we 
	have 
	$
	\varphi_{n} 
	= 
	\mu_{\bar{f}}( \varphi_{m} | \mathscr{F}_{\Lambda_{n}}) \
	\mu_{\bar{f}}
	$-a.s. 
	due to uniqueness of Radon-Nikodym derivative.
	
	We will show that this martingale converges
	in $L^{1}(\mu_{\bar{f}})$-norm to some 
	$\mathscr{F}$-measurable function $\varphi\geq 0$.
	To do this it is enough to prove that this 
	sequence is uniformly $\mu_{\bar{f}}$-integrable.
	Indeed, for any $K>1$ we have 
	\begin{align*}
	\int_{\Omega} \varphi_n 1_{\{\varphi_n\geq K\}}\, d\mu_{\bar{f}} 
	&\leq
	\frac{1}{\log K} 
	\int_{\Omega} [\varphi_n \log \varphi_n]\, 
	1_{\{\varphi_n\geq K\}}\, d\mu_{\bar{f}}  
	\\
	&<
	\frac{1}{\log K} 
	\int_{\Omega} [1+\varphi_n \log \varphi_n]\, 
	1_{\{\varphi_n\geq K\}}\, d\mu_{\bar{f}}  
	\\
	&\leq 
	\frac{1}{\log K}
	(1+\sup\{ \mathscr{H}_{ \Lambda_n}(\mu | \mu_{\bar{f}} )\ : n\in\mathbb{N} \}),
	\end{align*}
	where in the last inequality we used that $1+ x\log x\geq 0$, for all $x\geq 0$.
	From these estimates and the assumption that 
	$\sup\{ \mathscr{H}_{ \Lambda_n}(\mu | \mu_{\bar{f}} )\ : n\in\mathbb{N} \}<+\infty$,
	we have that $(\varphi_n)_{n\in\mathbb{N}}$ is a uniformly integrable martingale 
	and thereby convergent in $L^{1}(\mu_{\bar{f}})$.
	Therefore there exists $\varphi \in L^{1}(\mu_{\bar{f}})$ such that 
	$\lim_{n\to\infty} \int_{\Omega}|\varphi_n-\varphi|\, d\mu_{\bar{f}}=0$. 
	Of course, $\varphi = d\mu/d\mu_{\bar{f}}$. Proposition 3 of \cite{MR3377291}
	ensures that $\mu_{\bar{f}}$ is ergodic. Since $\mu$ is a shift invariant probability
	measure and $\mu\ll \mu_{\bar{f}}$ it follows from a classical result 
	in Ergodic Theory that $\mu=\mu_{\bar{f}}$ and the theorem is proved in this
	case. 
	
	Now we go back to the general case. We shall prove that
	for any fixed $n_0\in\mathbb{N}$, we have $\mu\gamma_{n_0} = \mu$.
	We split this proof in three steps.

	\textbf{Step 1}. For each $\delta>0$ and each $n_1>n_0$, there is $n_2\in\mathbb{N}$ 
	such that $n_0<n_1< n_2$ and 
	\begin{align}\label{eq1-dif-ent-rel}
	\mathscr{H}_{ \Lambda_{n_2}}(\mu | \mu_{\bar{f}} ) 
	-
	\mathscr{H}_{ \Lambda_{n_2}\setminus \Lambda_{n_1}  }(\mu | \mu_{\bar{f}} )
	\leq 
	\delta,
	\end{align}
	where 
	$\mathscr{H}_{ \Lambda_{0}}(\mu | \mu_{\bar{f}} )\equiv 0$.
	Indeed, let $m_1\in\mathbb{N}$ be such that if $n\geq \max\{n_1,m_1\}$, then 
	$n^{-1}\mathscr{H}_{ \Lambda_{n}}(\mu | \mu_{\bar{f}} )\leq \delta/n_1$.
	By taking $m_2=\lceil \max\{n_1,m_1\}/n_1\rceil$ we get  
	\begin{align*}
	\frac{1}{m_2} 
	\sum_{k=1}^{m_2}
	( 
	\mathscr{H}_{ \Lambda_{kn_1}}(\mu | \mu_{\bar{f}} ) 
	-
	\mathscr{H}_{ \Lambda_{(k-1)n_1}}(\mu | \mu_{\bar{f}} )
	)
	&\leq 
	\frac{1}{m_2} \mathscr{H}_{ \Lambda_{m_2n_1}}(\mu | \mu_{\bar{f}} ) 
	<\delta.
	\end{align*}
	Therefore there is some $n_2>n_1$ such that 
	$	\mathscr{H}_{ \Lambda_{n_2}}(\mu | \mu_{\bar{f}} ) 
	-
	\mathscr{H}_{ \Lambda_{n_2-n_1}}(\mu | \mu_{\bar{f}} )
	\leq 
	\delta
	$.
	Since $\mu,\mu_{f}\in\mathscr{M}_{\sigma}(\Omega)$ it follows that 
	$
	\mathscr{H}_{ \Lambda_{n_2-n_1}}(\mu | \mu_{\bar{f}} ) 
	= 
	\mathscr{H}_{ \Lambda_{n_2}\setminus \Lambda_{n_1}  }(\mu | \mu_{\bar{f}} )
	$
	and so the statement in step 1 is proved.

	\textbf{Step 2}. Let $\varphi_{n_2}$ be the function defined as in \eqref{def-phin}.
	Given $\varepsilon>0$ there exists $\delta>0$ such that 
	\[
	\int_{\Omega}
	\left| 
	\varphi_{n_2}
	-
	\frac{ d\mu|_{ \mathscr{F}_{\Lambda_{n_2}\setminus \Lambda_{n_1}  }  }  }
	{ d\mu_{\bar{f}}|_{ \mathscr{F}_{\Lambda_{n_2} \setminus \Lambda_{n_1}  }  }    }
	\right|
	\, d\mu_{\bar{f}} 
	< 
	\varepsilon
	\]
	whenever $n_1<n_2$ and 
	$
	\mathscr{H}_{ \Lambda_{n_2}}(\mu | \mu_{\bar{f}} ) 
	-
	\mathscr{H}_{ \Lambda_{n_2}\setminus \Lambda_{n_1}  }(\mu | \mu_{\bar{f}} )
	<
	\delta.
	$
	In fact, if $n_1$ and $n_2$ are two positive integers satisfying $n_1<n_2$.
	Then we have 
	\[
	\varphi_{n_2} = 0
	\quad
	\mu_{\bar{f}}-a.s.\ \text{on the set}\  
	\left\{ 	\frac{ d\mu|_{ \mathscr{F}_{\Lambda_{n_2}\setminus \Lambda_{n_1}  }  }  }
	{ d\mu_{\bar{f}}|_{ \mathscr{F}_{\Lambda_{n_2} \setminus \Lambda_{n_1}  }  }    }  
	= 0 
	\right\}
	\]
	because the last Radon-Nikodym derivative is equal to 
	$\mu_{\bar{f}}(\varphi_{n_2}|\mathscr{F}_{\Lambda_{n_2}\setminus \Lambda_{n_1}  })$. 
	Consider the function $\psi:[0,\infty)\to\mathbb{R}$, given by $\psi(x)=1-x+x\log x$.
	For some $0<r<\infty$, we have the following inequality $|1-x|\leq r\psi(x)+\varepsilon/2$, 
	for all $x\geq 0$. Therefore,
	\begin{align*}
	\int_{\Omega}&
	\left| 
	\varphi_{n_2}
	-
	\frac{ d\mu|_{ \mathscr{F}_{\Lambda_{n_2}\setminus \Lambda_{n_1}  }  }  }
	{ d\mu_{\bar{f}}|_{ \mathscr{F}_{\Lambda_{n_2} \setminus \Lambda_{n_1}  }  }    }
	\right|
	d\mu_{\bar{f}}
	\\
	&=
	\int_{\Omega}
	\frac{ d\mu|_{ \mathscr{F}_{\Lambda_{n_2}\setminus \Lambda_{n_1}  }  }  }
	{ d\mu_{\bar{f}}|_{ \mathscr{F}_{\Lambda_{n_2} \setminus \Lambda_{n_1}  }  }    }
	\left|
	1- 
	\varphi_{n_2}
	\left( \frac{ d\mu|_{ \mathscr{F}_{\Lambda_{n_2}\setminus \Lambda_{n_1}  }  }  }
	{ d\mu_{\bar{f}}|_{ \mathscr{F}_{\Lambda_{n_2} \setminus \Lambda_{n_1}  }  }    }
	\right)^{-1}
	\right|
	d\mu_{\bar{f}}
	\\
	&\leq 
	r
	\int_{\Omega}
	\frac{ d\mu|_{ \mathscr{F}_{\Lambda_{n_2}\setminus \Lambda_{n_1}  }  }  }
	{ d\mu_{\bar{f}}|_{ \mathscr{F}_{\Lambda_{n_2} \setminus \Lambda_{n_1}  }  }    }
	\; 	
	\psi\!\left( 
	\varphi_{n_2}
	\left( \frac{ d\mu|_{ \mathscr{F}_{\Lambda_{n_2}\setminus \Lambda_{n_1}  }  }  }
	{ d\mu_{\bar{f}}|_{ \mathscr{F}_{\Lambda_{n_2} \setminus \Lambda_{n_1}  }  }    }
	\right)^{-1}
	\right)\, 
	d\mu_{\bar{f}}
	+\frac{\varepsilon}{2}
	\\
	&=
	r
	\int_{\Omega}
	\varphi_{n_2}
	\log 
	\left[
	\varphi_{n_2}
	\left( \frac{ d\mu|_{ \mathscr{F}_{\Lambda_{n_2}\setminus \Lambda_{n_1}  }  }  }
	{ d\mu_{\bar{f}}|_{ \mathscr{F}_{\Lambda_{n_2} \setminus \Lambda_{n_1}  }  }    }
	\right)^{-1}
	\right]
	d\mu_{\bar{f}}
	+\frac{\varepsilon}{2}
	\\
	&= 
	r
	(\mathscr{H}_{ \Lambda_{n_2}}(\mu | \mu_{\bar{f}} ) 
	-
	\mathscr{H}_{ \Lambda_{n_2}\setminus \Lambda_{n_1}  }(\mu | \mu_{\bar{f}} ))
	+\frac{\varepsilon}{2}.
	\end{align*}
	By taking $\delta = \varepsilon/2r$ the step 2 statement is proved.

	\textbf{Step 3}. To prove that $\mu\gamma_{n_0} = \mu$ we fix a local function $\phi$
	(a bounded $\mathscr{F}_{\Lambda_n}$-measurable function for some $n\in\mathbb{N}$) 
	and $\varepsilon>0$. Since $\bar{f}$ is a H\"older potential it follows from 
	Theorem 3 of \cite{CL-rcontinuas-2016} that the family of probability kernels
	$(\gamma_n)_{n\geq 1}$ is quasilocal. Therefore there exists a local
	$\mathscr{F}_{\Lambda_{n_0}^c}$-measurable function $\tilde{\phi}$ such that 
	\[
	\sup_{y\in \Omega} \left| \tilde{\phi}(y) - \int_{\Omega}\phi(x) d\gamma_{n_0}(x|y)\right| 
	<
	\varepsilon.
	\]
	Let $n_1>n_0$ be such that $\phi$ is $\mathscr{F}_{\Lambda_{n_1}}$-measurable and
	$\tilde{\phi}$ is $\mathscr{F}_{\Lambda_{n_1}\setminus\Lambda_{n_0}}$-measurable. 
	Choose $\delta$ in terms of $\varepsilon$ as in Step 2, and define $n_2$ in terms
	of $n_1$ and $\delta$ as in Step 1. Then we have 
	\begin{align*}
	\left|
	\int_{\Omega} \phi \, d\mu\gamma_{n_0} - \int_{\Omega} \phi \, d\mu
	\right|
	&\leq
	\int_{\Omega} \left|\int_{\Omega}\phi(x) d\gamma_{n_0}(x|y) - \tilde{\phi}(y)\right| \, d\mu(y)
	\\
	&+
	\left|\int_{\Omega} \tilde{\phi}\, d\mu - \int_{\Omega} 
	\tilde{\phi}
	\frac{ d\mu|_{ \mathscr{F}_{\Lambda_{n_2}\setminus \Lambda_{n_1}  }  }  }
	{ d\mu_{\bar{f}}|_{ \mathscr{F}_{\Lambda_{n_2} \setminus \Lambda_{n_1}  }  }    }
	\, d\mu_{\bar{f}}\right|
	\\
	&+
	\int_{\Omega} 
	\frac{ d\mu|_{ \mathscr{F}_{\Lambda_{n_2}\setminus \Lambda_{n_1}  }  }  }
	{ d\mu_{\bar{f}}|_{ \mathscr{F}_{\Lambda_{n_2} \setminus \Lambda_{n_1}  }  }    }(y) 
	\left|\int_{\Omega}\phi(x) d\gamma_{n_0}(x|y) - \tilde{\phi}(y)\right|  
	d\mu_{\bar{f}}(y)
	\\
	&+
	\left|\int_{\Omega} 
	\frac{ d\mu|_{ \mathscr{F}_{\Lambda_{n_2}\setminus \Lambda_{n_1}  }  }  }
	{ d\mu_{\bar{f}}|_{ \mathscr{F}_{\Lambda_{n_2} \setminus \Lambda_{n_1}  }  }    }(y) 
	\left(\int_{\Omega}\phi(x) d\gamma_{n_0}(x|y) - \phi(y)\right)  
	d\mu_{\bar{f}}(y)
	\right|
	\\
	&+
	\|\phi\|_{\infty}
	\int_{\Omega} 
	\left|
	\varphi_{n_2}-
	\frac{ d\mu|_{ \mathscr{F}_{\Lambda_{n_2}\setminus \Lambda_{n_1}  }  }  }
	{ d\mu_{\bar{f}}|_{ \mathscr{F}_{\Lambda_{n_2} \setminus \Lambda_{n_1}  }  }    }
	\right| 
	d\mu_{\bar{f}}
	\\
	&+
	\left| 
	\int_{\Omega} \varphi_{n_2}\phi\,  d\mu_{\bar{f}}
	-\int_{\Omega} \phi\,  d\mu 
	\right|.
	\end{align*}
	Since $\tilde{\phi}$ is 
	$\mathscr{F}_{\Lambda_{n_2}\setminus \Lambda_{n_1} }$-measurable 
	and $\phi$ is $\mathscr{F}_{\Lambda_{n_2} }$-measurable,
	the second and last terms on rhs above are zero.
	Due to the choice of $\tilde{\phi}$, the first and third 
	terms are each at most $\varepsilon$.
	The fifth term is not larger than 
	$\varepsilon\|\phi\|_{\infty}$ because of our choice of $n_2$. 
	The fourth term is zero because of the DLR-equations. 
	Indeed, since $\mu_{\bar{f}}\in \mathcal{G}^{*}(\bar{f})$ and
	$\bar{f}$ is H\"older follow from Theorem 2 of \cite{CL-rcontinuas-2016} that 
	$\mu_{\bar{f}}\gamma_{n_0}=\mu_{\bar{f}}$, therefore 
	\begin{multline*}
	\left|\int_{\Omega} 
	\frac{ d\mu|_{ \mathscr{F}_{\Lambda_{n_2}\setminus \Lambda_{n_1}  }  }  }
	{ d\mu_{\bar{f}}|_{ \mathscr{F}_{\Lambda_{n_2} \setminus \Lambda_{n_1}  }  }    }
	(y)
	\left(\int_{\Omega}\phi(x) d\gamma_{n_0}(x|y) - \phi(y)\right)  
	d\mu_{\bar{f}}(y) \right|
	\\ 
	= 
	\left|
	\int_{\Omega}\int_{\Omega} 
	\frac{ d\mu|_{ \mathscr{F}_{\Lambda_{n_2}\setminus \Lambda_{n_1}  }  }  }
	{ d\mu_{\bar{f}}|_{ \mathscr{F}_{\Lambda_{n_2} \setminus \Lambda_{n_1}  }  }    }
	(x)\phi(x)
	\, d\gamma_{n_0}(x|y)d\mu_{\bar{f}}(y) 
	- \int_{\Omega} 
	\phi\,
	\frac{ d\mu|_{ \mathscr{F}_{\Lambda_{n_2}\setminus \Lambda_{n_1}  }  }  }
	{ d\mu_{\bar{f}}|_{ \mathscr{F}_{\Lambda_{n_2} \setminus \Lambda_{n_1}  }  }    }
	\, d\mu_{\bar{f}} \right|=0,
	\end{multline*}
	where in the first equality above we used that $\gamma_{n_0}$ 
	is a proper probability kernel and the DLR-equations. 	
\end{proof}

\section{Concluding Remarks}

In \cite{MR3377291} the classical notion of {\it equilibrium states},
when $\mathcal{A}$ is finite, is generalized for any compact metric 
alphabet. 
The authors fixed an a priori Borel probability measure $p$ on $\mathcal{A}$
and then defined equilibrium states for a H\"older potential $f$ as being 
a $\sigma$-invariant probability measure solving the 
variational problem
\[
\sup_{\mu\in\mathscr{M}_{\sigma}(\Omega) }
\mathtt{h}^{\!\mathtt{v}}(\mu)+\int_{\Omega} f\, d\mu.
\]
They shown that any probability measure $\mu$ satisfying 
$\mathscr{L}_{\bar{f}}^{*}\mu=\mu$ is a solution 
to the variational problem.
They also shown that there is a unique solution to 
$\mathscr{L}_{\bar{f}}^{*}\mu=\mu$, 
however they do not shown that any solution of the variational problem 
has to be of this form. We settled this question 
by combining the results of 
Theorems \ref{theoremdlr}, \ref{teo-equiv-entrop} 
and \ref{teo-uniqueness-EqStates-HolderCase}.  
We remark that when $\mathcal{A}$ is finite this result was first obtained in \cite{2015arXiv150801297G}.
Our techniques also provided the uniqueness 
of the equilibrium states for 
potentials in Walters space, which is defined as follows.
\begin{definition}[Walters Space]\label{def-walters-space}
	A potential
	$f:\Omega\to\mathbb{R}$ is said to be a Walters potential, 
	notation $W(\Omega,\sigma)$, 
	if the following condition is satisfied 
	\begin{align}\label{def-walters-class}
	\sup_{n\geq 1}\, 
	\sup_{\mathbf{a}\in \mathcal{A}^n}\, 
	|S_n(f)(\mathbf{a}x)-S_n(f)(\mathbf{a}y)|
	\to 0,
	\quad\text{when}\ d_{\Omega}(x,y)\to 0.
	\end{align}
\end{definition}

Theorem \ref{theoremdlr} ensures that 
\[
\lim_{n \rightarrow \infty} 
\frac{1}{n} \mathscr{H}_{\Lambda_n}(\mu | \nu)
=
\log\lambda_{f} - \int_{\Omega}f \, d\mu - \mathtt{h}^{\mathtt{s}}(\mu),
\]
where $f$ is an $\alpha$-H\"older potential,  $\mu\in \mathscr{M}_{\sigma}(\Omega)$ 
and $\nu\in \mathcal{G}^*(\bar{f})$. 
The way we computed the above limit, the regularity properties of 
$f$ are crucial.
The main steps in this computation is the existence of $\bar{f}$
(a normalized potential cohomologous to $f$),  
the identity \eqref{eq-free-energy}, and $o(n)$ estimates for \eqref{eq-aux1}. 
The two former conditions are also verified for any potential in the Walters class and 
the existence of $\bar{f}$ is proved in \cite{MR3538412}, 
for a general compact metric space alphabets.
Therefore  Theorem \ref{theoremdlr} can be generalized for such potentials.
In the case of finite alphabet, the theorem can also be generalized to potentials
in the Bowen class. 
Although it is not known whether 
the eigenfunctions associated to the maximal eigenvalue on this class is a continuous 
function its uniform upper and lower bounds, obtained in \cite{MR1783787},
are enough for our argument. For a general compact metric space alphabet, 
as far as we know, the existence of the maximal eigenfunction has not been 
proved.

To obtain the uniqueness of the equilibrium states for a Walters potential $f$,
we proceed as in the proof of 
Theorem \ref{teo-uniqueness-EqStates-HolderCase}, but using 
the equality $\mathcal{G}^{*}(\bar{f})= \mathcal{G}^{DLR}(\bar{f})$ 
and the uniquness of $\mathcal{G}^{*}(\bar{f})$ 
proved in \cite{CL-rcontinuas-2016}.

Since Theorem \ref{theoremdlr} 
is based on the maximal spectral data of the Ruelle operator we have  
a natural way to extend the definition of $\mathtt{h}(\mu|\nu)$, when $\nu$ 
is an equilibrium state of an arbitrary continuous potential $f$. 
For this extension instead of, using the maximal eigenvalue we 
use the spectral radius of $\mathscr{L}_{f}$ acting on $C(\Omega)$. 
So we can define $\mathtt{h}(\mu|\nu)$, where  
$\mu\in\mathscr{M}_{\sigma}(\Omega)$ and $\nu$ an 
equilibrium state for some continuous potential $f$, by putting   
\[
\mathtt{h}(\mu|\nu)
\equiv 
\log\rho(\mathscr{L}_{f}) - \int_{\Omega}f \, d\mu - \mathtt{h}^{\!\mathtt{v}}(\mu).
\]
Note that the above expression is similar to the one obtained in Theorem 
\ref{theoremdlr}, but the specific entropy $\mathtt{h}^{\mathtt{s}}(\mu)$ is replaced 
by $\mathtt{h}^{\!\mathtt{v}}(\mu)$ and we can see that  $\mathtt{h}(\mu|\nu)\geq 0$.
On the other hand, we do not know, in general,
whether $n^{-1} \mathscr{H}_{\Lambda_n}(\mu|\gamma_{\Lambda_n}(\cdot|y))$
converges to $\mathtt{h}(\mu|\nu)$. Counterexamples, on the lattice $\mathbb{Z}$, given in 
Section A.5.2 of \cite{MR1241537} (classical lattice systems) 
and in \cite{MR1653432} (quantum lattice systems) shows that 
this convergence can be a delicate issue. 

\section{Acknowledgements}
	The authors express many thanks to Aernout van Enter and Artur Lopes for 
	their valuable comments and references. L. Cioletti 
	and D. Aguiar are supported by CNPq and R. Ruviaro is supported by
	FAP-DF and FEMAT.

\end{document}